\documentclass[11pt, twoside, letterpaper]{amsart}
\usepackage{amsmath, hyperref, color}
\usepackage{leftidx}

\newtheorem{theorem}{Theorem}[section]
\newtheorem{corollary}[theorem]{Corollary}
\newtheorem{lemma}[theorem]{Lemma}
\newtheorem{proposition}[theorem]{Proposition}
\newtheorem{definition}[theorem]{Definition}
\newtheorem{remark}[theorem]{Remark}

\newtheorem{example}[theorem]{Example}

\numberwithin{equation}{section}

\def\R{\mathbb R}
\def\Pa{\mathbb P}
\def\Fcal{\mathcal F}
\def\Gcal{\mathcal G}

\def\Mcal{\mathcal M}

\def\proofof#1{\begin{proof}[Proof (of~#1).]}

\def\lbr{[\hspace{-0.05cm}[}
\def\rbr{]\hspace{-0.05cm}]}
\newcommand{\cadlag}{c\`{a}dl\`{a}g\ }
\newcommand{\bra}[1]{\left[#1\right]}
\newcommand{\indic}{\mathbb{I}}
\newcommand{\expec}{\mathbb{E}}

\begin{document}
\hfill{\today}\bigskip

\title[Characterization of max-continuous local martingales]{Characterization of max-continuous local martingales vanishing at infinity}

\author{B. Acciaio}
\address{The London School of Economics and Political Science}
%, Dept. Statistics, Houghton St, WC2A 2AE London, UK}
\email{b.acciaio@lse.ac.uk}

\author{I. Penner}
\address{Hochschule f\"ur Technik und Wirtschaft Berlin}
\email{Irina.Penner@HTW-Berlin.de}

\thanks{The authors would like to thank Anna Aksamit and Constantinos Kardaras for helpful discussions, as well as an anonymous referee for constructive comments that helped to improve the manuscript.}

%\subjclass{}
\keywords{Honest time; predictable stopping time; predictable set; local martingale; time of maximum}

\begin{abstract}
We provide a characterization of the family of non-negative local martingales that have continuous running supremum and vanish at infinity.
This is done by describing the class of random times that identify the times of maximum of such processes. In this way we extend to the case of general filtrations a result proved by Nikeghbali and Yor \cite{NikeghbaliYor06} for continuous filtrations. Our generalization is complementary to the one presented by Kardaras~\cite{Kardaras13}, and is obtained by means of similar tools.
\end{abstract}

\maketitle

\section{Statement of the main result}
\subsection{Set-up and notations}
We consider a filtered probability space $(\Omega, \Fcal, {\bf F}, \Pa)$, where ${\bf F}=(\Fcal_t)_{t\in\R_+}$ is a filtration satisfying the usual conditions of right-continuity and saturation by $\Pa$-null sets of $\Fcal:=\Fcal_\infty=\bigvee_{t\in\R_+}\Fcal_t$. For a process $X$, we write $\leftidx{^o}{X}$ and $\leftidx{^p}{X}$ to indicate the optional and predictable projections of $X$ (see Appendix~\ref{app.p}), and we
%We set $X_{0-}=0$ by convention, unless stated otherwise. 
use the notation $X_\infty:=\lim_{t\to\infty}X_t$ when the limit is well defined.\footnote{For processes for which the limit $\lim_{t\to\infty}X_t$ is well defined, we will consider the jump $\Delta X_\tau$ for general stopping times $\tau$, with the convention that $\Delta X_\tau\indic_{\{\tau=\infty\}}=0$.}
Inequalities between random variables and inclusions between subsets of $\Omega$ are always intended in the $\Pa$-a.s. sense. All (local) martingales are assumed to have \cadlag paths.

A \emph{random time} $\rho$ is a $[0,\infty]$-valued $\Fcal$-measurable map, and it is said to avoid all (predictable) stopping times if $\Pa\bra{\rho=\sigma}=0$ for every (predictable) stopping time $\sigma$. With this we also mean that $\Pa\bra{\rho=0}=0$. Furthermore, a random time $\rho$ is said to be the end of an optional (resp. predictable) set if there exists a set $\Gamma\subseteq[0,\infty]\times\Omega$ $\mathbf{F}$-optional (resp. predictable) such that $\rho=\sup\{t : (t,\omega)\in\Gamma\}$. 
To any random time $\rho$ we associate the so-called Az\'ema supermartingale, which is the optional projection of the stochastic interval $\indic_{\lbr0, \rho\lbr}$, i.e., the process $Z^\rho:=\leftidx{^o}{\left(\indic_{\lbr0, \rho\lbr}\right)}$; see \cite{Azema72}. Note that, for every $t\in\R_+$, $Z^\rho_t=\Pa\bra{\rho>t|\Fcal_t}$. The process $Z^\rho$ plays a key role in the enlargement of filtrations, as shown in \cite{Yor78}, \cite{Jeulin80a} and \cite{JeulinYor85}, as well as in the characterization of special classes of random times and in the theory of stochastic processes in general, see e.g. \cite{Nikeghbali06} and \cite{NikeghbaliPlaten12} and the references therein. Below we recall a prominent class of random times, which is the most studied within the non-stopping times.

\begin{definition}\label{def.ht}
A random time $\rho$ is called \emph{honest} if for every $t>0$, $\rho$ coincides with a $\Fcal_t$-measurable random variable on $\{\rho<t\}$.
\end{definition}
It is well-known that the end of an optional set is an honest time, and that a finite honest time is the end of an optional set; see \cite[p.74]{Jeulin80a}. Among honest times, widely studied in the literature are those which are the end of predictable sets. They can be identified through the corresponding Az\'ema supermartingale as follow, see \cite[Theorem~1.4]{Azema72} and \cite[Proposition~1]{JeulinYor78a}.
\begin{theorem}\label{thm.az72}
A random time $\rho$ satisfying $\Pa\bra{\rho>0}=1$ is the end of a predictable set if and only if $Z^\rho_{\rho-}=1$.
\end{theorem}
It was first noted in \cite{NikeghbaliYor06} that honest times can be related to times of maxima of local martingales; see also \cite{NikeghbaliPlaten08}, \cite{NikeghbaliPlaten12}, \cite{Kardaras13}.
In this note we shall provide a characterization of the following set of local martingales:
\[\Mcal_0:=\{\textrm{$L$ local martingale} : L\geq 0,\, L_0=1,\, L_\infty=0,\, L^*\, \textrm{is continuous}\},\]
where $L^*$ denotes the running supremum, that is, $L^*_t:=\sup_{0\leq s\leq t} L_s$ for $t\geq0$. Such characterization will be given in terms of the times of maximum of the processes in $\Mcal_0$, in the sense of the following definition, which is the same considered in \cite{Kardaras13}.
\begin{definition}\label{def:rhol}
For $L\in\Mcal_0$, we define the \emph{last time of maximum} of $L$ as
\[\rho_L:=\sup\{t>0 : L_{t-}=L^*_{t-}\}.\]
\end{definition}
Let $L\in\Mcal_0$. Note that $\Pa\bra{\rho_L<\infty}=1$ since $L_{\infty}=0$, and that $L_{\rho_L-}=L^*_{\rho_L-}=L^*_{\rho_L}$ by left-continuity of $L_-$ and continuity of $L^*$. Moreover, since $L^*$ is continuous, $L^*_\infty$ cannot be reached with a jump, and $\rho_L=\sup\{t>0 : L_{t-}=L^*_t\}=\sup\{t>0 : L_{t-}=L^*_\infty\}$, with $L_{\rho_L-}=L^*_\infty$. This motivates the name of \emph{last time of maximum}. In Corollary~\ref{cor.d} it is proved that $\rho_L$ is in fact the \emph{only} time of maximum of $L$ (i.e., the only time $\rho$ such that $L_{\rho-}=L_\infty^*$), hence the specification \emph{last} can be dropped.
The maximum is actually attained at $\rho_L$ (i.e., $L_{\rho}=L_\infty^*$) if $L$ has continuous paths, while this is in general not the case  if $L$ has jumps. In Corollary~\ref{cor.m} we show that $\rho_L$ can be expressed without need of taking left limits, that is,
\begin{equation}\label{rr}
\rho_L=\sup\{t>0 : L_{t}=L^*_{t}\},
\end{equation}
which is the definition considered in \cite{NikeghbaliYor06}.

\subsection{State of the art and main theorem}\label{sect:at}
The following characterization of $\Mcal_0$ is a consequence of~\cite{NikeghbaliYor06}.

\begin{theorem}\label{thm.ny}
Assume that every martingale is continuous. Then a random time $\rho$ satisfies $\rho=\rho_L$ for some local martingale $L\in\Mcal_0$ if and only if it is an honest time that avoids all stopping times. In this case, $Z^\rho=L/{L^*}$.
\end{theorem}
Without the continuity assumption, this characterization is no more true in general, in the sense that there exist local martingales $L$ in $\Mcal_0$ such that $\Pa\bra{\rho_L=\tau}>0$ for some stopping time $\tau$; see Example~\ref{ex.exp}.
Therefore, in order to obtain a similar characterization, one may either require some additional property on the local martingales, or alternatively consider a larger class of honest times.
The first approach is the one adopted in \cite{Kardaras13}, where the author considers those martingales in $\Mcal_0$ that do no jump at the time of maximum, obtaining the following characterization.
\begin{theorem}\label{thm.k}
A random time $\rho$ satisfies $\rho=\rho_L$ for some local martingale $L\in\Mcal_0$ such that $\Delta L=0\; \textrm{on}\; \{L_-=L^*_-\}$ if and only if it is an honest time that avoids all stopping times. In this case, $L_{\rho-}=L_{\rho}=L^*_\infty$, and $Z^\rho=L/{L^*}$.
\end{theorem}
On the other hand, in the present note we look for conditions that characterize the times of maxima for the processes in the class $\Mcal_0$, without any restriction. We shall therefore consider a larger class of honest times; see Remark~\ref{rmk.thm}. This leads to the main result of this note:
\begin{theorem}\label{thm.char}
For a random time $\rho$ the following are equivalent:
\begin{enumerate}
	\item[(I)] $\rho=\rho_L$ for one (and only one) local martingale $L\in\Mcal_0$;
	\item[(II)] $\rho$ is the end of a predictable set and it avoids all predictable stopping times.
\end{enumerate}
Under (any of) the previous conditions, $L_{\rho-}=L^*_\infty$, and $Z^\rho=L/{L^*}$.
\end{theorem} 

This multiplicative representation of the Az\'ema supermartingale expresses the fact that the ratio of a local martingale $L\in\Mcal_0$ over its running maximum equals the conditional probability that the time of maximum of $L$ is still ahead. This stresses the fact that by studying $Z^{\rho_L}$ one can infer information on $L$. In \cite{Ngen}, under the assumption that all martingales are continuous, this multiplicative decomposition has been extended to the framework of local submartingales of the class ($\Sigma$).

\begin{remark}\label{rmk.thm}
\it{1.}\ An honest time that avoids all stopping times is in particular the end of a predictable set, as shown in Lemma~\ref{rmk.ht}. This implies that the class of honest times we consider is in general larger than that in Theorems~\ref{thm.ny},\ref{thm.k}. Under the assumption that all martingales are continuous, however, the two classes coincide, since all stopping times are predictable. Therefore our Theorem~\ref{thm.char} is a genuine generalization of Theorem~\ref{thm.ny}. Note also that, in view of \eqref{rr}, Theorem~3.2 in \cite{NikeghbaliPlaten12} can be deduced from our characterization result.

\it{2.}\ Any stopping time $\tau$ is the end of the predictable set $\lbr0,\tau\rbr$. Therefore, the class of honest times considered in Theorem~\ref{thm.char} includes in particular all totally inaccessible stopping times $\tau$ such that $\Pa\bra{0<\tau<\infty}=1$. In this case we are also able to spell out the corresponding local martingales in the sense of Theorem~\ref{thm.char}, see the following example.
\end{remark}

\begin{example}\label{ex.exp}
Let $\tau$ be a totally inaccessible stopping time such that $\Pa\bra{0<\tau<\infty}=1$. Consider $Y:=(I_{\lbr0,\tau\lbr})^p$, the predictable compensator of $I_{\lbr0,\tau\lbr}$, which is a continuous non-increasing process, and the compensated martingale $M:=I_{\lbr0,\tau\lbr}-Y$.
Then the process defined by
\[
L:=\mathcal{E}(M)=\exp(-Y)I_{\lbr0,\tau\lbr}                                                                                                                                                                                                                                                                                                                                                                                                                                                                                                                                                                                                                                                                                                                                                                                                                                                                                                                                                                                                                                                                                                                                                                                                                                                                                                                                                                                                                                                                                                                                                                                                                                                                                                                                                                                                                                                                                                                                                                                                                                                                                                                                                                                                                                                                                                                                                                                                                                                                                                                                                                                                                                                                                                                                                                                                                                                                                                                                                                                                                                                                                                                                                                                                                                                                                                                                                                                                                                                
\]
is a local martingale that belongs to $\Mcal_0$ and has maximum at $\tau$, in the sense that $\rho_L=\tau$.

Note that, since $L$ has a down jump at the time of maximum, it is excluded from the processes considered in Theorem~\ref{thm.k}; see also \cite[Remark~1.4]{Kardaras13}.
\end{example}

\begin{remark}
The request that $\rho$ is the end of a predictable set in Theorem~\ref{thm.char}-(II) cannot be weakened by asking $\rho$ to be the end of an optional set, i.e. $\rho$ being an honest time, as the following example shows. 
\end{remark}

\begin{example}\label{ex.end_pr}
Let $N$ be a Poisson process and $M$ its compensated martingale. Define the process $S:=\mathcal{E}(M)$ and the random time $\rho:=\sup\{t : S_t=S_t^*\}$. Since $S_t$ goes to zero as $t$ goes to infinity, $\rho$ is a finite honest time. Moreover, $S_\rho=S^*_\rho=S^*_\infty$.
Now, the Az\'ema supermartingale associated to $\rho$ satisfies
\begin{equation}\label{eq.chin}
Z^{\rho}_t=\Pa[\rho>t|\Fcal_t]
=\Pa[\exists s>t, S_s\geq S_t^*|\Fcal_t]
=\Pa\Big[\sup_{s\geq t}S_s\geq S^*_t\Big|\Fcal_t\Big]
\leq\frac{S_t}{S_t^*},\quad t\in\R_+,
\end{equation}
cf. Section~4.2.3 in \cite{ACDJ_ex14}.
Indeed, to see the third equality, note that $\{\sup_{s\geq t}S_s\geq S^*_t\}=\{\exists s>t, S_s\geq S_t^*\}\cup\{S_s<S^*_t\; \forall s>t, S_t= S^*_t\}$. Since $\{S_s<S^*_t\; \forall s>t, S_t= S^*_t\}=\{\rho=t\}$, and $\Pa[\rho=t]=0$, the stated equality holds.
The inequality in \eqref{eq.chin} follows by \eqref{conditional_Doob_tau}.
Therefore we have
\[
Z^{\rho}_{\rho-}\leq\frac{S_{\rho-}}{S_{\rho-}^*}<1,
\]
where the strict inequality follows from the nature of the process $S$: $S$ has positive jumps at the same times where $N$ jumps, and it is decreasing between jump times, which implies $S_{\sigma-}<S^*_{\sigma-}$ at every jump time $\sigma$. Since $\rho$ is a finite random time, with $\lbr\rho\rbr\subset\{\Delta S\neq 0\}$, then $S_{\rho-}<S^*_{\rho-}$ holds, as claimed.
In particular, $\rho$ is the end of an optional set but not the end of a predictable set, from Theorem~\ref{thm.az72}.
\end{example}

We conclude the section with an observation regarding the modeling of default times in the hazard-rate approach of credit risk, where the Az\'ema supermartingale already proved to play an important role, see e.g. \cite{EJY00}, \cite{JYC09} and \cite{CN12}. According to this approach, a default time is defined as a random time, say $\tau$, which is not a stopping time with respect to some initial filtration $\bf F$. The whole information available in the market is then described by the filtration ${\bf F}^\tau=(\Fcal^\tau_t)_{t\in\R_+}$ obtained by progressively enlarging $\bf F$ with $\tau$, that is, $\bf{F}^\tau$ is the smallest (right-continuous) filtration that contains $\bf F$ and makes $\tau$ a stopping time:
\[
\Fcal_t^\tau=\Gcal_{t+},\quad \textrm{with}\quad \Gcal_t=\Fcal_t\vee\sigma(\tau\wedge t).
\]
In this approach the default time occurs as a surprise for the market, hence $\tau$ needs to be an $\bf{F}^\tau$-totally inaccessible stopping time. This is the case for the random times identified in Theorem~\ref{thm.char}, as shown in Remark~\ref{rem:h}, implying that the random times we are studying are suitable for modeling default times in the hazard-rate approach of credit risk.

\section{Proof of Theorem~\ref{thm.char}}
In Sections~\ref{sect:pr},~\ref{sect:md},~\ref{sect:db} we show some preliminary results that will be used in order to prove Theorem~\ref{thm.char} in Section~\ref{sect:proof}.

\subsection{Dual predictable and optional projections}\label{sect:pr}
In this section we consider random times $\rho$ such that $\Pa\bra{0<\rho<\infty}=1$, noting that this condition is satisfied by the random times intervening in Theorem~\ref{thm.char}.
We denote by $A$ and $a$ the dual optional and predictable projections of the process $H=\indic_{\lbr\rho,\infty\lbr}$, respectively (see Appendix~\ref{app.dp}).
The Az\'ema supermartingale associated to $\rho$ has then the following additive decompositions
\begin{equation}\label{eq.dm}
Z^{\rho}_t=\expec[A_\infty|\Fcal_t]-A_t=\expec[a_\infty|\Fcal_t]-a_t,
\end{equation}
see Appendix~\ref{app.ad}.
\begin{lemma}\label{lemma.cts}
For a random time $\rho$ such that $\Pa\bra{0<\rho<\infty}=1$ the following hold:
\begin{itemize}
	\item[(a)] $\rho$ avoids all stopping times if and only if $A$ ($\equiv a$) is continuous;
	\item[(b)] $\rho$ avoids all predictable stopping times if and only if $a$ is continuous.
\end{itemize}
\end{lemma}
\begin{proof}
Since the increasing process $H$ has a unique jump of size one at time $\rho$, for any stopping time $\tau$ we have $\Delta H_\tau=\indic_{\{\rho=\tau\}}$. Now Theorem VI.$76$ in \cite{DellacherieMeyer80} yields
\begin{eqnarray*}
\Delta A\hspace*{-0.2cm}&=&\hspace*{-0.2cm}\leftidx{^o}{\left(\indic_{\lbr\rho\rbr}\right)},\ \textrm{i.e.}\ \Delta A_\tau=\expec[\Delta H_\tau|\Fcal_\tau]=\Pa[\rho=\tau|\Fcal_\tau]\;\;\textrm{for all stopping times $\tau$},\\
\Delta a\hspace*{-0.2cm}&=&\hspace*{-0.2cm}\leftidx{^p}{\left(\indic_{\lbr\rho\rbr}\right)},\ \textrm{i.e.}\ \Delta a_\tau=\expec[\Delta H_\tau|\Fcal_{\tau-}]=\Pa[\rho=\tau|\Fcal_{\tau-}]\;\;\textrm{for all predictable stopping times $\tau$},
\end{eqnarray*}
from which the statements of the lemma follow.
\end{proof}

\begin{lemma}\label{rmk.ht}
An honest time that avoids all stopping times is the end of a predictable set.
\end{lemma}
\begin{proof}
Let $\rho$ be an honest time. By definition of dual optional projection we have
\[
\expec\bra{\int_{[0,\infty)}\leftidx{^o}{\left(\indic_{\lbr\rho,\infty\lbr}\right)}d\indic_{\lbr\rho,\infty\lbr}}=
\expec\bra{\int_{[0,\infty)}\indic_{\lbr\rho,\infty\lbr}dA}=
\expec\bra{A_\infty-A_{\rho-}}.
\]
Now, let $\tilde{Z}^{\rho}$ be the optional projection of the stochastic interval $\indic_{\lbr0, \rho\rbr}$, and recall that $\Pa\bra{\tilde{Z}^{\rho}_\rho=1}=1$ since $\rho$ is a finite honest time, see  %\cite[Proposition~1]{JeulinYor78a} and 
\cite[Proposition~5.1]{Jeulin80a}.
This gives
\[
\expec\bra{\int_{[0,\infty)}\leftidx{^o}{\left(\indic_{\lbr\rho,\infty\lbr}\right)}d\indic_{\lbr\rho,\infty\lbr}}=
\expec\bra{\leftidx{^o}{\left(\indic_{\lbr\rho,\infty\lbr}\right)}_{\rho}}=
\expec\bra{1-Z^{\rho}_\rho}=
\expec\bra{\tilde{Z}^{\rho}_\rho-Z^{\rho}_\rho}=
\expec\bra{\Delta A_\rho}=
\expec\bra{A_\rho-A_{\rho-}},
\]
showing that $\Pa\bra{A_t=A_{t\wedge\rho}}=1$.
If in addition $\rho$ avoids all stopping times, then $A\equiv a$ by Lemma~\ref{lemma.cts}. This implies $a_t=a_{t\wedge\rho}$, which by \cite[Proposition 3]{JeulinYor78a} is equivalent to the fact that $\rho$ is the end of a predictable set, as claimed.
\end{proof}

\begin{remark}\label{rem:h}
With the notation introduced at the end of Section~\ref{sect:at}, we notice that any random time $\rho$ that avoids all ($\bf F$-) predictable stopping times, is a totally inaccessible stopping time in the enlarged filtration $\bf{F}^\rho$.
Indeed, by Lemma~\ref{lemma.cts}, the dual predictable projection of $H$ is continuous. On the other hand, the compensator of $H$ in $\bf{F}^\rho$ is given by the process $\alpha_t=\int^{t\wedge\rho}_0\frac{da_s}{Z^{\rho}_{s-}}$, by Remark~4.5-(3) in \cite{Jeulin80a}. Therefore $\alpha$ is continuous, or equivalently, $H$ is quasi-left continuous, which means that $\rho$ is a totally inaccessible stopping time in $\bf{F}^\rho$, as claimed.
\end{remark}

\subsection{Multiplicative decomposition of the Az\'ema supermartingale}\label{sect:md}
As proved in \cite{iw65}, any nonnegative supermartingale can be written as the product of a local martingale and a non-increasing process. The decomposition is in general not unique. One can however require conditions on the factorizing processes in order to identify some particular decomposition. This is what we want to do for the
Az\'ema supermartingale associated to a random time. In \cite{Kardaras13} it is noted that the {\it optional} multiplicative decomposition given in \cite{Kardaras12} is useful for the characterization of honest times. Here we adopt a similar argumentation, using instead the {\it predictable} multiplicative decomposition given in \cite{PennerReveillac13}. Precisely, by applying \cite[Proposition 4.5]{PennerReveillac13} to the dual predictable projection of $\indic_{\lbr\rho, \infty\lbr}$, we obtain the following:

\begin{proposition}\label{prop.pr}
Let $\rho$ be a random time and $a$ the dual predictable projection of $\indic_{\lbr\rho, \infty\lbr}$. Then there exists a pair of adapted \cadlag processes $(L,D)$ such that:
\begin{enumerate}
	\item $L$ is a non-negative local martingale with $L_0=1$;
	\item $D$ is a non-increasing predictable process with $0\leq D\leq 1$;
	\item\label{ad} $a_t=-\int_{[0,t]}L_{s-}dD_s$ for all $t\in\R_+$, with $L_{0-}:=1$;
	\item\label{mult.dec} $Z^{\rho}_t=L_tD_t,\; t\in\R_+$;
	\item\label{LDz} $L_t=L_0+\int_0^t\indic_{\{D_s>0\}}dL_s$, $D_t=D_0+\int_0^t\indic_{\{L_{s-}>0\}}dD_s,\; t\in\R_+$.
\end{enumerate}
\end{proposition}

As the continuity of the dual optional projection $A$ of $\indic_{\lbr\rho,\infty\lbr}$ is of major importance in \cite{Kardaras13}, so the continuity of the dual predictable projection $a$ will be in proving the main results of this note. Indeed, note that when $a$ is continuous, then $D$ is continuous as well, which is crucial in the proof of Proposition~\ref{prop.d}. This shows how in the present setting we cannot simply use the multiplicative decomposition considered in \cite{Kardaras13}, based on the process $A$, since in our case $A$ is no more continuous. However, in the particular case of $\rho$ avoiding all stopping times, then $a\equiv A$ and the decomposition in Proposition~\ref{prop.pr}-\eqref{mult.dec} coincides with that in \cite{Kardaras13}.

\begin{proposition}\label{prop.d}
Let $\rho$ be a random time that avoids all predictable stopping times, and let $(L,D)$ be a pair of processes associated to $\rho$ as in Proposition~\ref{prop.pr}. Then $D_\rho$ has standard uniform distribution under $\Pa$.
\end{proposition}

\begin{proof}
%This implication is proved in the same way as in \cite[Lemma~3.4]{Kardaras13}.
For $u\in(0,1]$, we define the stopping time $\xi_u:=\inf\{t\in\R_+ | D_t<u\}$, with the usual convention $\inf\emptyset=+\infty$.
By assumption $\Pa(\rho=0)=0$, and therefore $Z^{\rho}_0=1$, which in turn implies $D_0=1$ by Proposition~\ref{prop.pr}.
By Lemma~\ref{lemma.cts}-(b), $a$ is continuous, hence $D$ is continuous as well, which in turn gives $D_{\xi_u}=u$ on $\{\xi_u<\infty\}$ for all $u\in(0,1]$. Now fix $v\in(0,1]$. By definition of dual predictable projection, and by Proposition~\ref{prop.pr}-\eqref{ad} together with the continuity of $D$, we have
\[
\Pa\bra{D_{\rho}\ge v}=\expec\Big[\int_0^\infty\indic_{\{D_t\ge v\}}da_t\Big]=-\expec\Big[\int_0^\infty\indic_{\{D_t\ge v\}}L_{t}dD_t\Big].
\]
By means of $(\xi_u)_{u\in(0,1]}$ we perform a change-of-time (see e.g. \cite[Ch.1]{BNS10}), obtaining
\[
-\expec\Big[\int_0^\infty\indic_{\{D_t\ge v\}}L_{t}dD_t\Big]
=\expec\Big[\int_{D_\infty}^1\indic_{\{D_{\xi_u}\geq v\}}L_{\xi_u}du\Big].
\]
Now, as in the proof of Lemma~3.4 in \cite{Kardaras13}, we 
note that $\Pa\bra{\rho=\infty}=0$, since $\rho$ avoids all predictable stopping times, which gives $0=Z^{\rho}_\infty=L_\infty D_\infty$. This in turn implies $L_\infty=0$ on $\{D_\infty>0\}$, hence $L_{\xi_u}\indic_{\{\xi_u<\infty\}}=L_{\xi_u}$, since $\{\xi_u=\infty\}\subseteq\{D_\infty>0\}$ holds for $u\in(0,1]$. Therefore
\[
\expec\Big[\int_{D_\infty}^1\indic_{\{D_{\xi_u}\geq v\}}L_{\xi_u}du\Big]
=\expec\Big[\int_{D_\infty}^1\indic_{\{D_{\xi_u}\geq v\}}\indic_{\{\xi_u<\infty\}}L_{\xi_u}du\Big]
=\expec\Big[\int_{D_\infty\vee v}^1\indic_{\{\xi_u<\infty\}}L_{\xi_u}du\Big],
\]
where in the second equality we used the fact that $D_{\xi_u}=u$ on $\{\xi_u<\infty\}$.
We proceed by noticing that $\xi_u=\infty$ on $\{D_\infty>u\}$, which yields
\[
\expec\Big[\int_{D_\infty\vee v}^1\indic_{\{\xi_u<\infty\}}L_{\xi_u}du\Big]
=\expec\Big[\int_v^1\indic_{\{\xi_u<\infty\}}L_{\xi_u}du\Big]
=\expec\Big[\int_v^1L_{\xi_u}du\Big],
\]
again using the equality
$L_{\xi_u}\indic_{\{\xi_u<\infty\}}=L_{\xi_u}$.
Finally, for $u\in(0,1]$, since $Z^{\rho}\leq 1$ and $D\geq u$ on $\lbr0,\xi_u\rbr$, then $L\leq 1/u$ holds on $\lbr0,\xi_u\rbr$ by Proposition~\ref{prop.pr}-\eqref{mult.dec}. This implies that $L^{\xi_u}$ is a true martingale, hence, in particular, $\expec[L_{\xi_u}]=1$. Therefore, by Fubini's theorem we get
\[
\expec\Big[\int_v^1 L_{\xi_u}du\Big]
=\int_v^1\expec[L_{\xi_u}]du=1-v.
\]
Altogether we proved that $\Pa\bra{D_{\rho}\ge v}=1-v$ for $v\in(0,1]$, which shows that $D_\rho$ has standard uniform distribution.
\end{proof}

\subsection{Doob's maximal identity}\label{sect:db}
Doob's maximal identity states that, for $L\in\Mcal_0$, $1/L_\infty^*$ has uniform distribution. This is a consequence of the optional stopping theorem and has been known for some time, see e.g. \cite[Exercise~II.3.12]{RY99}. In \cite{NikeghbaliYor06} this identity is exploited for the first time in connection with the multiplicative decomposition of the Az\'ema supermartingale and used to derive the conditional distribution of $L_\infty^*$; see also \cite{Ngen} for an extension to the framework of local submartingales of the class ($\Sigma$) in a context where all martingales are continuous.
A generalization of the Doob's maximal identity has then been proved in \cite{Kardaras13}, in the form of Lemma~\ref{lemma_Doob}. This plays a crucial role in proving our main result, and below we provide a proof of it based on the pathwise super-replication of the digital option with payoff $\Phi:=\indic_{\{L^*_\infty\geq x\}}$, for $x\in(1,\infty)$. It corresponds to finding an upper bound on the no-arbitrage price of $\Phi$, and the exact price under $\Pa$ when $L\in\Mcal_0$, see Remark~\ref{rmk_Doob}.

\begin{lemma}\label{lemma_Doob}
Let $L$ be a nonnegative local martingale with $L_0 = 1$. Then $\Pa\bra{L_\infty^*\geq x}\leq 1/x$ holds for all $x\in(1,\infty)$. Furthermore, $\Pa\bra{L_\infty^*\geq x}= 1/x$ holds for all $x\in(1,\infty)$ if and only if $L\in\Mcal_0$.
\end{lemma} 

\begin{proof}%[Proof of Lemma~\ref{lemma_Doob}]
Let $L$ be a nonnegative local martingale with $L_0=1$. For $x\in(1,\infty)$, define the stopping time $\tau_x:=\inf\{t>0 : L_t> x\}$. Then the following pathwise inequalities trivially hold
\begin{equation}\label{doob.in}
\indic_{\{L^*_\infty> x\}}=\indic_{\{\tau_x<\infty\}}\leq L_{\tau_x}/x,
\end{equation}
and by taking expectations we obtain $\Pa\bra{L_\infty^*> x}\leq \expec\bra{L_{\tau_x}}/x\leq 1/x$, for all $x\in(1,\infty)$. Clearly, this also gives $\Pa\bra{L_\infty^*\geq x}\leq 1/x$, for all $x\in(1,\infty)$.
Now note that the inequality in \eqref{doob.in} is indeed an equality for all $x\in(1,\infty)$ if and only if the two following facts hold: $L_{\tau_x}=x$ on $\{\tau_x<\infty\}$ for all $x\in(1,\infty)$, and $L_\infty=0$ on $\cup_{x\in(1,\infty)}\{\tau_x=\infty\}=\Omega$. These conditions hold $\Pa$-a.s. if and only if $L\in\Mcal_0$. 
In this case $L^{\tau_x}$ is a true martingale, since bounded, and passing to the expectations gives $\Pa\bra{L_\infty^*> x}= \expec\bra{L_{\tau_x}}/x= 1/x$ for all $x\in(1,\infty)$. This concludes the proof.
\end{proof}

\begin{remark}\label{rmk_Doob}
With the notation introduced above, the inequality in \eqref{doob.in} means that the option $\Phi$ is pathwise super-replicated by following 
the simple strategy which consists in buying $1/x$ shares of $L$ at time $0$, and selling them when (if) the stock price exceeds the value $x$ (that is, at time $\tau_x$ if $\tau_x<\infty$). Clearly, having equality in \eqref{doob.in} corresponds to exact pathwise replication.
Note that \eqref{doob.in} can be seen as a particular case of the pathwise inequality given in \cite{MR1839367}, where calls are used to hedge digital options.
Being $L_0=1$, the above strategy has initial cost $1/x$, and it is an admissible strategy in the sense of \cite{MR1304434}, since the corresponding portfolio value is uniformly bounded from below.
Now, by no-arbitrage arguments, the capital $1/x$ needed to set up this portfolio cannot be smaller than the price of $\Phi$. On the other hand, for $L\in\Mcal_0$ the portfolio value of the above strategy is also uniformly bounded from above, hence selling this portfolio is admissible as well, and so the option $\Phi$ is constrained to have the same price $1/x$ of the replicating portfolio.
\end{remark}

Let $L$ be a nonnegative local martingale with $L_0 = 1$. Then, for any stopping time $\tau$, on the set $\{\tau<\infty\}$ we have 
\begin{equation}\label{conditional_Doob_tau}
\Pa\bra{\sup_{t\geq\tau}L_t \geq L_\tau^* \Big| \mathcal{F}_\tau}\leq \frac{L_\tau}{L_\tau^*},
\end{equation}
with equality holding if $L^*$ is continuous on $\lbr\tau,\infty\lbr$ and $L_\infty=0$, in which case we have
\begin{equation}\label{conditional_Doob_eq_tau}
\Pa\bra{\sup_{t\geq\tau}L_t > L_\tau^* \Big| \mathcal{F}_\tau}=\Pa\bra{\sup_{t\geq\tau}L_t \geq L_\tau^* \Big| \mathcal{F}_\tau}=\frac{L_\tau}{L_\tau^*}.
\end{equation}
Indeed, this is clearly true on $\{L_\tau=0\}$, and on $\{L_\tau>0\}$ it follows by applying Lemma~\ref{lemma_Doob} to the local martingale $\left(\frac{L_{\tau+t}}{L_\tau}\right)_{t\geq0}$ in the filtration $(\Fcal_{\tau+t})_{t\geq0}$.

\begin{corollary}\label{cor.d}
Let $L\in\Mcal_0$ and let $\rho$ be a time of maximum for $L$, in the sense that $\Pa\bra{L_{\rho-}=L^*_\infty}=1$. Then $\Pa\bra{\rho=\rho_L}=1$, that is, the time of maximum is unique. Moreover, the Az\'ema supermartingale corresponding to $\rho$ is given by $Z^{\rho}=L/L^*$.
\end{corollary}

\begin{proof}
The proof follows from \eqref{conditional_Doob_eq_tau} by the same arguments used in \cite[Lemma~3.6]{Kardaras13}.
\end{proof}

\begin{corollary}\label{cor.m}
For $L\in\Mcal_0$, define the random time
\[\rho'_L:=\sup\{t>0 : L_{t}=L^*_{t}\}.\]
Then $\Pa\bra{\rho_L=\rho'_L}=1$.
\end{corollary}
\begin{proof}
We first prove the inequality $\rho_L\leq\rho'_L$. Recall that $L_{\rho_L-}=L^*_\infty$ (see the discussion after Definition~\ref{def:rhol}), and define the set $\Lambda:=\{L_{\rho'_L-}=L^*_\infty\}$. 
Note that $\rho_L=\rho'_L$ on $\Lambda$. Indeed, suppose this is not true and consider the random time $\sigma:=\rho'_L\indic_\Lambda+\rho_L\indic_{\Lambda^c}$. Since $L_{\rho'_L-}=L^*_\infty=L_{\rho_L-}$ on $\Lambda$, then $\sigma$ is a time of maximum in the sense that $L_{\sigma-}=L^*_\infty$, which gives a contradiction by Corollary~\ref{cor.d}.
On the other hand, on $\Lambda^c$ we have $L_{\rho'_L-}<L^*_\infty$, and then $L_{\rho'_L}=L^*_\infty$ must hold, since $\rho'_L=\sup\{t>0 : L_{t}=L^*_\infty\}$.
Therefore, on $\Lambda^c$ we have $L_{\rho'_L-}<L^*_\infty=L_{\rho'_L}$, which means that $L$ jumps to its overall maximum $L^*_\infty$ at time $\rho'_L$. Since $L^*$ is continuous, this implies the existence of a random time $\rho''$ such that $\rho''<\rho'_L$ and $L_{\rho''-}=L^*_\infty=L^*_{\rho'_L}$ hold on $\Lambda^c$. Then $\rho''=\rho_L$ on $\Lambda^c$, again by Corollary~\ref{cor.d}. Therefore, $\rho_L=\rho''<\rho'_L$ on $\Lambda^c$, which concludes the proof of $\rho_L\leq\rho'_L$.

Now, in \cite{NikeghbaliPlaten12} it is shown that the Az\'ema supermartingale associated to $\rho'_L$ satisfies $Z^{\rho'_L}=L/{L^*}$. Therefore, from Theorem~\ref{thm.char} we get that the optional projection $\leftidx{^o}{(\indic_{\lbr\rho_L, \rho'_L\lbr})}$ is equal to zero up to indistinguishability. This gives $\Pa\bra{\rho_L=\rho'_L}=1$.
\end{proof}

\subsection{Proof of Theorem~\ref{thm.char}}\label{sect:proof}

(I) $\Rightarrow$ (II): Let $L\in\Mcal_0$. By definition, $\rho_L$ is the end of the predictable set $\{L_-=L^*_-\}$ and $\Pa\bra{\rho_L<\infty}=1$. It remains to prove that $\rho_L$ avoids all predictable stopping times. To this end, fix a predictable stopping time $\tau$ such that $\Pa[\tau<\infty]=1$, and define $B:=\{L_{\tau-}=L^*_{\tau-}\}\in\Fcal_{\tau-}$. We denote by $\tau_B$ the restriction given by $\tau_B=\tau$ on $B$ and $\tau_B=\infty$ elsewhere. Recall that $\Delta L_{\tau_B}\indic_{\{\tau_B=\infty\}}=0$ by convention. By Proposition~I.2.10 in \cite{JS03}, $\tau_B$ is predictable, which in turn yields $\expec[\Delta L_{\tau_B}|\Fcal_{\tau_B-}]=0$. On the other hand, on $\{\Delta L_{\tau_B}>0\}$ we have $L^*_{\tau_B}>L^*_{\tau_B-}$, which is impossible because $L^*$ is continuous. Consequently, $\Pa[\Delta L_{\tau_B}\leq 0]=1$, which in turn implies $\Pa\bra{\Delta L_{\tau_B}=0}=1$. Now we proceed as in the proof of $(2)\Rightarrow(1)$ in \cite[Theorem~1.2]{Kardaras13} and note that, since $L_{\rho_L-}=L^*_{\rho_L-}=L^*_{\rho_L}$ (see the discussion after Definition~\ref{def:rhol}), we have $\{\rho_L=\tau, L_\tau<L_\tau^*\}\subseteq\{L_{\tau-}=L_{\tau-}^*, \Delta L_\tau<0\}=\{\tau_B<\infty, \Delta L_{\tau_B}<0\}$. From $\Pa\bra{\Delta L_{\tau_B}=0}=1$, we then deduce that $\Pa\bra{\rho_L=\tau, L_\tau<L_\tau^*}=0$. On the other hand, on $\{L_\tau=L_\tau^*\}$ we have $\Pa\bra{\sup_{t>\tau}L_t>L_\tau^*|\Fcal_\tau}=\Pa\bra{\sup_{t\geq\tau}L_t>L_\tau^*|\Fcal_\tau}=\frac{L_\tau}{L_\tau^*}=1$, by \eqref{conditional_Doob_eq_tau}, which gives $\Pa\bra{\rho_L=\tau, L_\tau=L_\tau^*}=0$.
It follows that $\Pa\bra{\rho_L=\tau}=0$, as wanted.
\smallskip

(II) $\Rightarrow$ (I): Let $\rho$ be the end of a predictable set such that it avoids all predictable stopping times, and let $(L,D)$ be a pair of processes associated to $\rho$ as in Proposition~\ref{prop.pr}.
We first prove that $\rho$ is a time of maximum for $L$, in the sense that $\Pa\bra{L_{\rho-}=L^*_\infty}=1$, and that $L\in\Mcal_0$.
Since $\rho$ is the end of a predictable set and $\Pa\bra{\rho>0}=1$, $Z^{\rho}_{\rho-}=1$ follows by Theorem~\ref{thm.az72} and, by Proposition~\ref{prop.pr}-\eqref{mult.dec}, $1=Z^{\rho}_{\rho-}=L_{\rho-}D_{\rho-}=L_{\rho-}D_\rho$ holds. This implies that $1/L_{\rho-}$ has uniform distribution, by Proposition~\ref{prop.d}. On the other hand, from Lemma~\ref{lemma_Doob}, $1/L^*_\infty$ dominates a uniform random variable in the sense of first order stochastic dominance. This means that $X:=1/L^*_\infty$ dominates $Y:=1/L_{\rho-}$ in first order stochastic dominance, and since $\Pa\bra{X\leq Y}=1$, then $\Pa\bra{X=Y}=1$ follows, that is, $\Pa\bra{L_{\rho-}=L^*_\infty}=1$. This in turn implies that $1/L^*_\infty$ has uniform distribution, hence $L\in\Mcal_0$ by Lemma~\ref{lemma_Doob}. Now, $\rho=\rho_L$ and $Z^{\rho}=L/{L^*}$ follow from Corollary~\ref{cor.d}.
We are left to show that $\rho$ identifies a unique $L\in\Mcal_0$. In order to do so, we proceed as in \cite{NikeghbaliYor06} and apply It\^o's formula to $Z^{\rho}=L/{L^*}$, obtaining the following Doob-Meyer decomposition of $Z^{\rho}$:
\[Z^{\rho}_t=\mathbb{E}[\log L^*_\infty|\Fcal_t]-\log L_t^*=1+\int_0^t\frac{1}{L^*_s}dL_s-\log L_t^*,\]
by \eqref{eq.dm}.
The uniqueness of the Doob-Meyer decomposition implies that if $L,M\in\Mcal_0$ are such that $\rho_L=\rho_M$, then $L=M$ holds up to indistinguishability.
\qed

\appendix
\section{Basic notions}
In this section we recall the notions of optional (resp. predictable) projection and of dual optional (resp. predictable) projection, which play a fundamental role in this note.
\subsection{The optional and predictable projections}\label{app.p}
Let $X$ be a non-negative measurable process.
The \emph{optional projection} of $X$ is the unique (up to indistinguishability) optional process $Y$ such that
\[
\expec\bra{X_\tau\indic_{\{\tau<\infty\}}|\Fcal_\tau}=Y_\tau\indic_{\{\tau<\infty\}}\;\, \textrm{a.s.}
\]
for every stopping time $\tau$.

The \emph{predictable projection} of $X$ is the unique (up to indistinguishability) predictable process $N$ such that
\[
\expec\bra{X_\tau\indic_{\{\tau<\infty\}}|\Fcal_{\tau-}}=N_\tau\indic_{\{\tau<\infty\}}\;\, \textrm{a.s.}
\]
for every predictable stopping time $\tau$.

\subsection{The dual optional and predictable projections}\label{app.dp}

Let $H$ be an integrable raw increasing process.
The \emph{dual optional projection} of $H$ is the optional increasing process $U$ defined by
\[
\expec\bra{\int_{[0,\infty)}X_tdU_t}=
\expec\bra{\int_{[0,\infty)}X_tdH_t}
\]
for any bounded optional process $X$.

The \emph{dual predictable projection} of $H$ is the predictable increasing process $V$ defined by
\[
\expec\bra{\int_{[0,\infty)}X_tdV_t}=
\expec\bra{\int_{[0,\infty)}X_tdH_t}
\]
for any bounded predictable process $X$.

\subsection{The additive decompositions in \eqref{eq.dm}}\label{app.ad}
Let $A$ and $a$ denote the dual optional and predictable projections of the process $H_t:=\indic_{\{\rho\leq t\}}$, respectively. From the previous section we have that, for every bounded predictable process $X$,
\begin{equation}\label{eq.ap}
\expec\bra{X_\rho}=
\expec\bra{\int_{[0,\infty)}X_tdA_t}=\expec\bra{\int_{[0,\infty)}X_tda_t}.
\end{equation}
Fix $t\in\R_+$ and $F_t$ bounded $\Fcal_t$-measurable, then $X_s:=F_t1_{\{t<s\}}$ defines a predictable process, and from \eqref{eq.ap} we have
\[
\expec\bra{F_t\indic_{\{t<\rho\}}}=\expec\bra{F_t(A_\infty-A_t)}=\expec\bra{F_t(a_\infty-a_t)}.
\]
This in turn yields
\[
Z^{\rho}_t=\expec\bra{\indic_{\{t<\rho\}}|\Fcal_t}=\expec\bra{A_\infty|\Fcal_t}-A_t=\expec\bra{a_\infty|\Fcal_t}-a_t,
\]
which is exactly equation \eqref{eq.dm}.

\bibliographystyle{alpha}

\providecommand{\bysame}{\leavevmode\hbox to3em{\hrulefill}\thinspace}
\providecommand{\MR}{\relax\ifhmode\unskip\space\fi MR }

\providecommand{\MRhref}[2]{%
  \href{http://www.ams.org/mathscinet-getitem?mr=#1}{#2}
}
\providecommand{\href}[2]{#2}

\end{document}